\documentclass{dcds-bOF}
\usepackage{amsmath,amsthm,amssymb,xcolor,bbm,mathrsfs}
  \usepackage{paralist}
  \usepackage{graphics} 
  \usepackage{epsfig} 
\usepackage{graphicx}  \usepackage{epstopdf}
 \usepackage[colorlinks=true]{hyperref}
\hypersetup{urlcolor=blue, citecolor=red}

\def\currentvolume{X}
 \def\currentissue{X}
  \def\currentyear{200X}
   \def\currentmonth{XX}
    \def\ppages{X--XX}
     \def\DOI{10.3934/dcdsb.2021084}

\newtheorem{theorem}{Theorem}
\newtheorem{lemma}{Lemma}
\newtheorem{definition}{Definition}
\newtheorem{property}{Property}
\newtheorem{example}{Example}
\newtheorem{corollary}{Corollary}

\newcommand{\n}{\hspace*{-6pt}}

\DeclareMathOperator{\diag}{diag}
\DeclareMathOperator{\dist}{dist}
\DeclareMathOperator{\diam}{diam}
\DeclareMathOperator{\Span}{Span}

\title[Mixed HK dynamics] 
      {Mixed Hegselmann-Krause dynamics}

\author[Hsin-Lun Li]{}


 \email{hsinlunl@asu.edu}

\begin{document}
\maketitle

\centerline{\scshape Hsin-Lun Li}
\medskip
{\footnotesize
 \centerline{School of Mathematical and Statistical Sciences}
   \centerline{Arizona State University, Tempe, AZ 85287, USA}

} 

\bigskip

 \centerline{(Communicated by Chris Cosner)}

\begin{abstract}
 The original Hegselmann-Krause (HK) model consists of a set of~$n$ agents that are characterized by their opinion, a number in~$[0, 1]$.
 Each agent, say agent~$i$, updates its opinion~$x_i$ by taking the average opinion of all its neighbors, the agents whose opinion differs from~$x_i$ by at most~$\epsilon$.
 There are two types of~HK models: the synchronous~HK model and the asynchronous~HK model.
 For the synchronous model, all the agents update their opinion simultaneously at each time step, whereas for the asynchronous~HK model, only one agent chosen uniformly at
 random updates its opinion at each time step.
 This paper is concerned with a variant of the~HK opinion dynamics, called the mixed~HK model, where each agent can choose its degree of stubbornness and mix its opinion
 with the average opinion of its neighbors at each update.
 The degree of the stubbornness of agents can be different and/or vary over time.
 An agent is not stubborn or absolutely open-minded if its new opinion at each update is the average opinion of its neighbors, and absolutely stubborn if its opinion does not
 change at the time of the update.
 The particular case where, at each time step, all the agents are absolutely open-minded is the synchronous~HK model.
 In contrast, the asynchronous model corresponds to the particular case where, at each time step, all the agents are absolutely stubborn except for one agent chosen uniformly
 at random who is absolutely open-minded.
 We first show that some of the common properties of the synchronous~HK model, such as finite-time convergence, do not hold for the mixed model.
 We then investigate conditions under which the asymptotic stability holds, or a consensus can be achieved for the mixed model.
\end{abstract}

\section{Introduction}
 The Hegselmann-Krause (HK) model is a popular opinion dynamics model describing the interactions among a population of agents.
 In the standard~HK model, there are~$n$ agents and each agent updates its opinion by taking the average opinion of its neighbors.
 More precisely, let
 $$ x_i (t + 1) = \frac{1}{|N_i(t)|} \sum_{j \in N_i (t)} x_j (t) \quad \hbox{where} \quad x_i (t) \in \mathbf{R^d} $$
 represents the opinion of agent~$i$ at time~$t \in \mathbf{N}$, and let
 $$ N_i (t) = \{j \in [n] : \|x_i (t) - x_j (t) \| \leq \epsilon \} \quad \hbox{where} \quad [n] = \{1, 2, \ldots, n \} $$
 be the set of agents whose opinion differs from the opinion of agent~$i$ by at most~$\epsilon$, that we call the neighbors of agent~$i$ at time~$t$.
 Here,~$\|\ \, \|$ refers to the Euclidean norm and~$\epsilon$ is a positive number that represents a confidence bound.
 The authors of~\cite{1} considered the one-dimensional modified~HK model as follows:
 $$ x_i (t + 1) = \alpha_i \,x_i (t) + \frac{(1 - \alpha_i)}{|N_i (t)|} \sum_{j \in N_i (t)} x_j (t) \quad \hbox{where} \quad x_i (t) \ \hbox{and} \ \alpha_i \in [0, 1]. $$
 In words, the convex combination indicates that agent~$i$ mixes its opinion with the average opinion of its neighbors, with the parameter~$\alpha_i$ measuring
 the degree of stubbornness of agent~$i$.
 In this paper, we extend the modified~HK model to higher dimensional sets of opinions and allow the degree of stubbornness~$\alpha_i$ to vary over time.
 The resulting model can be expressed in matrix form as
\begin{equation}
\label{mHK}
 x (t + 1) = \diag (\alpha (t)) \,x (t) + (I - \diag (\alpha (t))) \,A (t) \,x (t)
\end{equation}
 where $A (t) \in \mathbf{R^{n\times n}}$ is row stochastic with
 $$ A_{ij} = \mathbbm{1} \{j \in N_i (t) \} / |N_i(t)| $$
 and where
 $$ \begin{array}{rclcl}
      x (t) & \n = \n & (x_1 (t), x_2 (t), \ldots, x_n (t))' & \n = \n & \hbox{transpose of} \ (x_1 (t), x_2 (t), \ldots, x_n (t)), \vspace*{4pt} \\
      \alpha (t) & \n = \n & (\alpha_1 (t), \alpha_2 (t), \ldots, \alpha_n (t))' & \n = \n & \hbox{transpose of} \ (\alpha_1 (t), \alpha_2 (t), \ldots, \alpha_n (t)). \end{array} $$
 In particular, agent~$i$ is absolutely stubborn when~$\alpha_i (t) = 1$ and absolutely open-minded when~$\alpha_i (t) = 0$.
 Observe also that~\eqref{mHK} reduces to
\begin{itemize}
 \item the synchronous HK model if~$\alpha (t) = \Vec{0}$ for all $t \geq 0$ and \vspace*{2pt}
 \item the asynchronous HK model if~$\alpha (t) = (\mathbbm{1} \{j \neq i (t) \})_{j = 1}^n$ for all~$t\geq0$ and for some~$i (t) \in [n]$ chosen uniformly at random.
\end{itemize}
 Our main objective is to study the strategies the agents should play so that the asymptotic stability holds, or a consensus can be achieved.
 Some of the common properties of the synchronous~HK model do not hold for the mixed~HK model.
 Before going into the details, we need the following definitions.
\begin{definition}{\rm
 An \emph{opinion profile} at time~$t$ or simply a \emph{profile} at time~$t$ is an undirected graph~$\mathscr{G} (t)$ with the vertex set and edge set
 $$ \mathscr{V} (t) = [n] \quad \hbox{and} \quad \mathscr{E} (t) = \{ij : i\neq j \ \hbox{and} \ \|x_i (t) - x_j (t) \| \leq \epsilon \}. $$}
\end{definition}
 Apart from \cite{2}, the opinion profile is simple.
\begin{definition}{\rm
 The \emph{termination time} of~$n$ agents, $T_n$, is the maximum number of iterations in~\eqref{mHK} by reaching a steady state over all initial profiles, i.e.,
 $$ T_n = \inf \{t\geq0 : x (t) = x (s) \ \hbox{for all} \ s \geq t \}. $$}
\end{definition}
\begin{definition}{\rm
 The \emph{convex hull} generated by~$v_1, v_2, \ldots, v_n \in \mathbf{R^d}$ is the smallest convex set containing~$v_1, v_2, \ldots, v_n$, i.e.,
 $$ C (\{v_1, v_2 \ldots, v_n \}) = \{v : v = \sum_{i = 1}^n \lambda_i v_i \ \hbox{where} \ (\lambda_i)_{i = 1}^n \ \hbox{is stochastic} \}. $$}
\end{definition}
\begin{definition}{\rm
 A profile $\mathscr{G}(t)$ is \emph{$\mathbf{\delta}$-trivial} if any two of its vertices are at a distance of at most~$\delta$ apart.
 In particular, $\mathscr{G}(t)$ is complete if it is $\epsilon$-trivial.}
\end{definition}
\begin{definition}{\rm
 For~$\delta > 0$, $x(t)$ in~\eqref{mHK} is a \emph{$\mathbf{\delta}$-equilibrium} if there is a partition
 $$ \{G_1, G_2, \ldots, G_m \} \ \hbox{of the set} \ \{x_1 (t), x_2 (t), \ldots, x_n (t) \} $$
 such that the following two conditions hold:
 $$ \dist (C (G_i), C (G_j)) > \epsilon \ \hbox{for all} \ i \neq j \quad \hbox{and} \quad \diam (C (G_i)) \leq \delta \ \hbox{for all} \ i \in [m]. $$}
\end{definition}
\begin{definition}{\rm
 A \emph{merging time} is a time~$t$ that two agents with different opinions at time~$t - 1$ have the same opinion at time~$t$, i.e.,
 $$ x_i (t) = x_j (t) \quad \hbox{and} \quad x_i (t - 1) \neq x_j (t - 1) \quad \hbox{for some} \quad i, j \in [n]. $$}
\end{definition}
 The following are some properties distinct from the synchronous HK model.
\begin{property}
 The termination time is not finite.
\end{property}
\begin{example}{\rm
 Assume that~$n = 2, \ d = 1$,
 $$ x_1 (0) = 0, \ x_2 (0) = \epsilon \ \hbox{and} \ \alpha_1 (t) = \alpha_2 (t) = 1/2 \ \hbox{for all} \ t \geq 0. $$
 Then, at each time step, $x_1$ and $x_2$ get closer to each other.
 However, never do they reach a steady state in finite time.}
\end{example}
\begin{property}
 Agents merging at time~$t$ may depart at time~$t + 1$.
 In particular, $\mathscr{G}(t)$ $\epsilon$-trivial may not imply that~$x (t + 1)$ in~\eqref{mHK} is a steady state.
\end{property}
\begin{example}{\rm
 Assume that~$n = 3, \ d = 2$,
 $$ \begin{array}{rclrclrcl}
      x_1 (0) & \n = \n & (0, 0), \qquad & \alpha_1 (0) & \n = \n & 0, \qquad & \alpha_1 (1) & \n = \n & 1/3, \vspace*{4pt} \\
      x_2 (0) & \n = \n & (\epsilon, 0), \qquad & \alpha_2 (0) & \n = \n & 0, \qquad & \alpha_2 (1) & \n = \n & 1/2, \vspace*{4pt} \\
      x_3 (0) & \n = \n & (\epsilon/2, \epsilon). \end{array} $$
 Then, $x_1$ and $x_2$ merge at time $t = 1$ but depart at time $t = 2$.}
\end{example}
\begin{property}
 A $\delta$-equilibrium may not exist for all $0<\delta \leq \epsilon$.
\end{property}
\begin{example}{\rm
 Assume that~$n = 3, \ d = 2$,
 $$ x_1 (0) = (0, 0), \ x_2 (0) = (\epsilon, 0), \ x_3 (0) = (\epsilon/2, \epsilon) \ \hbox{and} \ \alpha_1 (t) = \alpha_2 (t) = 1/2 $$
 for all $t \geq 0$.
 Then, $x$ has no $\delta$-equilibrium for all~$0 < \delta \leq \epsilon$.
 Note that vertex 3 of the profile is isolated all the time.}
\end{example}
 The following lemma plays an important role in the proof of the main theorems.
\begin{lemma}
\label{L1}
 Let~$\lambda_1, \ldots, \lambda_n \in \mathbf{R}$ with~$\sum_{i = 1}^n \lambda_i = 0$ and~$x_1, \ldots, x_n \in \mathbf{R^d}$. Then, for
 $$ \lambda_1 x_1 + \lambda_2 x_2 + \cdots + \lambda_n x_n, $$
 the terms with positive coefficients can be matched with the terms with negative coefficients in the sense that
 $$ \sum_{i = 1}^n \,\lambda_i x_i = \sum_{i, c_i \geq 0, j, k \in [n]} c_i (x_j - x_k) \quad \hbox{and} \quad \sum_{i} \,c_i = \sum_{j, \lambda_j \geq 0} \lambda_j. $$
\end{lemma}
\begin{proof}
 We prove the result by induction on~$n$.
 Without loss of generality, we may assume that~$\lambda_1 \geq \lambda_2 \geq \cdots \geq \lambda_n$.
 For~$n = 2$, $\lambda_1 + \lambda_2 = 0$ implies that
 $$ \lambda_2 = - \lambda_1 \ \hbox{and} \ \lambda_1 \geq 0 \quad \hbox{therefore} \quad \lambda_1 x_1 + \lambda_2 x_2 = \lambda_1 (x_1 - x_2), $$
 which proves the result for~$n = 2$.
 Now, assume that~$n > 2$.
 Because the~$\lambda_i$'s add up to 0, we have~$\lambda_n \leq 0$. Define
 $$ \lambda_n = - \lambda \quad \hbox{and} \quad i = \min \bigg\{m \in \mathbf{Z^+} : \sum_{k = 1}^m \lambda_k \geq \lambda \bigg \}. $$
 Then, $\lambda_k \geq 0$ for all~$1 \leq k \leq i$ so
 $$ \begin{array}{rcl}
    \displaystyle \sum_{k = 1}^n \,\lambda_k x_k & \n = \n &
    \displaystyle \sum_{k = 1}^{i - 1} \,\lambda_k (x_k - x_n) + \bigg(\lambda - \sum_{k = 1}^{i - 1} \lambda_k \bigg)(x_i - x_n) \vspace*{4pt} \\ && \hspace*{10pt} + \
    \displaystyle \bigg(\sum_{k = 1}^i \,\lambda_k - \lambda \bigg) x_i + \sum_{k = i+1}^{n - 1} \lambda_k x_k. \end{array} $$
 Now, observe that~$\lambda - \sum_{k = 1}^{i - 1} \lambda_k \geq 0$, $\sum_{k = 1}^i \lambda_k - \lambda \geq 0$ and
 $$ \bigg(\sum_{k = 1}^i \,\lambda_k - \lambda \bigg) + \sum_{k = i + 1}^{n - 1} \lambda_k = \sum_{k = 1}^{n - 1} \lambda_k - \lambda = \sum_{k = 1}^{n - 1} \lambda_k + \lambda_n =
    \sum_{k = 1}^n \,\lambda_k = 0. $$
 By the induction hypothesis,
 $$ \begin{array}{rcl}
    \displaystyle \bigg(\sum_{k = 1}^i \,\lambda_k - \lambda \bigg) x_i + \sum_{k = i + 1}^{n - 1} \lambda_k x_k & \n = \n &
    \displaystyle \sum_{\ell, c_\ell \geq 0, j, k \in [n - 1] - [i - 1]} c_\ell \,(x_j - x_k), \vspace*{4pt} \\
    \displaystyle \sum_{\ell} \,c_\ell & \n = \n &
    \displaystyle \bigg(\sum_{k = 1}^i \,\lambda_k - \lambda \bigg) + \sum_{k \in [n - 1] - [i], \lambda_k \geq 0} \lambda_k. \end{array} $$
 Hence,~$\lambda_1 x_1 + \lambda_2 x_2 + \cdots + \lambda_n x_n$ can be written as
 $$ \sum_{\ell, \hat{c}_\ell \geq 0, j, k \in [n]} \hat{c}_\ell (x_j - x_k) $$
 where the sum of the coefficients~$\hat c_\ell$ is given by
 $$ \begin{array}{rcl}
    \displaystyle \sum_\ell \,\hat{c}_\ell & \n = \n &
    \displaystyle \sum_{k = 1}^{i - 1} \,\lambda_k + \bigg(\lambda - \sum_{k = 1}^{i - 1} \,\lambda_k \bigg) + \sum_\ell \,c_\ell =
    \displaystyle \lambda + \sum_\ell \,c_\ell \vspace*{4pt} \\ & \n = \n &
    \displaystyle \lambda + \sum_{k = 1}^i \,\lambda_k - \lambda + \sum_{k \in [n - 1] - [i], \lambda_k \geq 0} \lambda_k =
    \displaystyle \sum_{k \in [n - 1], \lambda_k \geq 0} \lambda_k = \sum_{k \in [n], \lambda_k \geq 0} \lambda_k. \end{array} $$
 This completes the proof.
\end{proof}
This result allows us to observe the interactions among the agents and derive a better upper bound.
 For any~$x, y \in C (\{v_1,...,v_n \})$,
\begin{itemize}
 \item the coefficients of all~$v_i$'s in $x - y$ add up to zero and \vspace*{2pt}
 \item the sum of the positive coefficients of the $v_i$'s in $x - y$ is at most one.
\end{itemize}
 In particular, by Lemma~\ref{L1} and the triangle inequality,
 $$ \|x - y \| \leq \max_{i, j \in [n]} \|v_i - v_j \| \leq \diam (C (\{v_1, \ldots, v_n \})) $$
 therefore~$\diam (C (\{v_1, \ldots, v_n \})) = \max_{i, j \in [n]} \|v_i - v_j \|$.
\begin{lemma}
\label{L2}
 We have
 $$ \diam (C (\{v_1, \ldots, v_n \})) = \max_{i, j \in [n]} \|v_i - v_j \| \ \hbox{for all} \ v_i \in \mathbf{R^d}. $$
\end{lemma}
 In contrast with the synchronous HK model, $\mathscr{G} (t)$ $\epsilon$-trivial may not imply that a consensus is reached at the next time step.
 However, $\mathscr{G} (t + 1)$ is again~$\epsilon$-trivial.
 Observe that
 $$ x_i (t + 1) \in C (\{x_1 (t), \ldots, x_n (t) \}) \quad \hbox{for all} \quad i \in [n] $$
 and according to Lemma \ref{L2},
 $$ \max_{i, j \in [n]} \|x_i (t + 1) - x_j (t + 1) \| \leq \max_{i, j \in [n]} \|x_i (t) - x_j (t) \|. $$
\begin{lemma}[$\delta$-trivial-preserving]
\label{L3}
 For any~$\delta > 0$, if
$$ \mathscr{G} (t) \ \hbox{is~$\delta$-trivial}, \ \hbox{then} \ \mathscr{G} (t + 1) \ \hbox{is~$\delta$-trivial}. $$
\end{lemma}
 Indeed, we can derive a better upper bound for~$\|x_i (t + 1) - x_j (t + 1) \|$ by re-organizing the terms of~$x_i (t + 1) - x_j (t + 1)$.
\begin{lemma}
\label{L4}
 Assume that~$\mathscr{G} (t)$ is~$\epsilon$-trivial. Then,
 $$ \begin{array}{l}
    \displaystyle \max_{i, j \in [n]} \|x_i (t + 1) - x_j (t + 1) \| \vspace*{0pt} \\ \hspace*{20pt} \leq
    \displaystyle \max_{i, j \in [n], \alpha_i (t) \geq \alpha_j (t)} \bigg(\alpha_i (t) - \frac{\alpha_i (t) - \alpha_j (t)}{n} \bigg) \max_{i, j \in [n]} \|x_i (t) - x_j (t) \|. \end{array} $$
\end{lemma}
\begin{proof}
 Let $x = x (t)$, $x' = x (t + 1)$ and~$\alpha = \alpha (t)$.
 For any $i, j \in [n]$ with~$\alpha_i \geq \alpha_j$,
 $$ x_i' - x_j' = \bigg(\alpha_i - \frac{\alpha_i - \alpha_j}{n} \bigg) x_i - \bigg(\alpha_j + \frac{\alpha_i - \alpha_j}{n} \bigg) x_j -
                  \frac{\alpha_i - \alpha_j}{n} \sum_{k \in [n] - \{i, j \}} x_k. $$
 Observe that
 $$ \alpha_i - \frac{\alpha_i - \alpha_j}{n} \geq \alpha_i - \frac{\alpha_i}{n} \geq 0, \quad
    \alpha_j + \frac{\alpha_i - \alpha_j}{n} \geq 0 \quad \hbox{and} \quad
    \frac{\alpha_i - \alpha_j}{n} \geq 0, $$
 showing that~$x_i$ is the only term with nonnegative coefficient, whereas the other terms have nonpositive coefficients.
 Because~$x_i'\in C (\{x_1, x_2, \ldots, x_n\})$ for all~$i \in [n]$, it follows from~Lemma~\ref{L1} that
 $$ x_i' - x_j' = \bigg(\alpha_j + \frac{\alpha_i - \alpha_j}{n} \bigg)(x_i - x_j) + \frac{\alpha_i - \alpha_j}{n} \sum_{k \in [n] - \{i, j \}} (x_i - x_k) $$
 and the coefficients of the terms~$x_i - x_k$ for $k \in [n] - \{i\}$ add up to~$\alpha_i - \frac{\alpha_i - \alpha_j}{n}$.
 Thus, by the triangle inequality,
 $$ \begin{array}{rcl}
    \|x_i' - x_j' \| & \n \leq \n &
    \displaystyle \bigg(\alpha_j + \frac{\alpha_i - \alpha_j}{n} \bigg) \|x_i - x_j \| + \frac{\alpha_i - \alpha_j}{n} \sum_{k \in [n] - \{i, j \}} \|x_i - x_k \| \vspace*{4pt} \\ & \n \leq \n & \displaystyle \bigg(\alpha_i-\frac{\alpha_i-\alpha_j}{n} \bigg) \max_{k \in [n] - \{i \}} \|x_i - x_k \| \vspace*{4pt} \\ & \n = \n &
    \displaystyle \bigg(\alpha_i-\frac{\alpha_i-\alpha_j}{n} \bigg) \max_{k \in [n]} \|x_i - x_k \| \vspace*{4pt} \\ & \n \leq \n &
    \displaystyle \max_{i, j \in [n], \alpha_i \geq \alpha_j} \bigg(\alpha_i - \frac{\alpha_i - \alpha_j}{n} \bigg) \max_{i, k \in [n]} \|x_i - x_k \|. \end{array} $$
 If $\alpha_i \leq \alpha_j$, then exchanging the roles of~$i$ and~$j$, we get
 $$ \begin{array}{rcl}
    \|x_j' - x_i' \| & \n \leq \n &
    \displaystyle \max_{j, i \in [n], \alpha_j \geq \alpha_i} \bigg(\alpha_j - \frac{\alpha_j - \alpha_i}{n} \bigg) \max_{j, k \in [n]} \|x_j - x_k \| \vspace*{4pt} \\ & \n = \n &
    \displaystyle \max_{i, j \in [n], \alpha_i \geq \alpha_j} \bigg(\alpha_i - \frac{\alpha_i - \alpha_j}{n} \bigg) \max_{i, k \in [n]} \|x_i - x_k \|. \end{array} $$
 In conclusion,
 $$ \max_{i, j \in [n]} \|x_i' - x_j' \| \leq \max_{i, j \in [n], \alpha_i \geq \alpha_j} \bigg(\alpha_i - \frac{\alpha_i - \alpha_j}{n} \bigg) \max_{i, k \in [n]} \|x_i - x_k \|. $$
 This completes the proof.
\end{proof}

 Observe that
 $$ \beta_t := \max_{i, j \in [n], \alpha_i (t) \geq \alpha_j (t)} \bigg(\alpha_i (t) - \frac{\alpha_i (t) - \alpha_j (t)}{n} \bigg) \leq 1. $$
 Therefore, $\mathscr{G} (t)$ $\epsilon$-trivial implies~$\mathscr{G} (s)$ $\epsilon$-trivial for all $s\geq t$. Hence,
 $$ \max_{i, j \in [n]} \|x_i (s + 1) - x_j (s + 1) \| \leq \beta_s \max_{i, j \in [n]} \|x_i (s) - x_j (s) \| \quad \hbox{for all} \quad s\geq t. $$
\begin{theorem}
\label{T1}
 Assume that~$\limsup_{t \to \infty} \beta_t < 1$ and that~$\mathscr{G} (t)$ is~$\epsilon$-trivial. Then,
 $$ \lim_{t \to \infty} \max_{i, j \in [n]} \|x_i (t) - x_j (t) \| = 0. $$
\end{theorem}
\begin{proof}
 Define
 $$ d_s = \max_{i, j \in [n]} \|x_i (s) - x_j (s) \|. $$
 According to Lemma \ref{L4},
 $$ \mathscr{G} (t) \ \hbox{$\epsilon$-trivial} \quad \Longrightarrow \quad d_{s + 1} \leq \beta_s d_s \ \hbox{for all} \ s \geq t. $$
 Since $\limsup_{t \to \infty} \beta_t < 1$, there exists~$(t_i)_{i = 1}^{\infty} \subset \mathbf{N}$ strictly increasing with~$t_1 \geq t$ such
 that~$\beta_{t_i} \leq \delta < 1$ for some~$\delta$ and for all~$i \geq 1$.
 For any~$s > t_1$, we have~$t_{i_s} < s\leq t_{i_s + 1}$ for some $i_s \in \mathbf{Z^+}$ therefore
 $$ d_s \leq \beta_{s - 1} \beta_{s - 2} \cdots \beta_{t_1} d_{t_1} \leq \delta^{i_s} d_{t_1}. $$
 As $s \to \infty$, $i_s \to \infty$. Thus,
 $$ \limsup_{s \to \infty} d_s \leq 0, $$
 showing that the limit exists.
 This completes the proof.
\end{proof}

 In an~$\epsilon$-trivial profile, agents need not be open-minded all the time.
 As long as there are infinitely many~$\beta_t$ with an upper bound less than one, eventually will the population reach a consensus.
 The next theorem shows that, even though the profile is not~$\epsilon$-trivial, still can the agents' opinions converge.
\begin{theorem}
\label{T2}
 Define~$d_t^i = \max_{j \in N_i (t)} \|x_i (t) - x_j (t) \|$. If
 $$ \sum_{t = 0}^{\infty} \ (1 - \alpha_i (t)) \bigg(1 - \frac{1}{|N_i (t)|} \bigg) d_t^i < \infty, \ \hbox{then} \ x_i (t) \to x_i \in \mathbf{R^d} \ \hbox{as} \ t \to \infty. $$
\end{theorem}
\begin{proof}
 By Lemma \ref{L1} and the triangle inequality,
 $$ \begin{array}{l}
    \|x_i (t) - x_i (t + 1) \| =
    \displaystyle \|(1 - \alpha_i (t)) \bigg(1 - \frac{1}{|N_i (t)|} \bigg) x_i (t) - \frac{1 - \alpha_i (t)}{|N_i(t)|} \sum_{j \in N_i (t) -\{ i\}} x_j (t) \| \vspace*{4pt} \\ \hspace*{25pt} =
    \displaystyle \frac{1 - \alpha_i(t)}{|N_i(t)|} \ \bigg\|\sum_{j \in N_i (t) - \{i\}}[x_i (t) - x_j (t)] \bigg\| \leq
    \displaystyle (1 - \alpha_i (t)) \bigg(1 - \frac{1}{|N_i (t)|} \bigg) d_t^i, \end{array} $$
 from which it follows that
 $$ \sum_{t = 0}^{\infty} \ \|x_i (t) - x_i (t + 1) \| < \infty. $$
 This shows that~$(x_i (t))_{t = 0}^{\infty}$ is a Cauchy sequence in~$\mathbf{R^d}$. Hence,~$x_i (t)$ converges to some $x_i$ in $\mathbf{R^d}$ as $t$ goes to infinity.
 This completes the proof.
\end{proof}

 The assumption of Theorem~\ref{T2} is difficult to check because it depends on the entire dynamics' trajectory.
 However, since~$\alpha_i (t)$ is controllable and
 $$ \bigg(1 - \frac{1}{|N_i(t)|} \bigg) \,d^i_t $$
 is bounded, the assumption holds if the sum of~$1 - \alpha_i (t)$ over time is finite.
 For instance, given~$a > 1$, if
 $$ 1 - \alpha_i (t) = O \bigg(\frac{1}{t^a} \bigg), \  \hbox{then}  \ x_i (t) \ \hbox{converges to some} \ x_i \in \mathbf{R^d} \ \hbox{as} \ t \to \infty. $$
 Next, we study several conditions under which every component of a profile is~$\delta$-trivial in finite time or under which the asymptotic stability holds.
 The following definition and lemmas will lead us to these conditions.
\begin{definition}{\rm
 A symmetric matrix~$M$ is called a \emph{generalized Laplacian} of a graph $G = (V, E)$ if for~$x, y \in V$, the following two conditions hold:
 $$ M_{xy} = 0 \ \hbox{for} \  x \neq y \ \hbox{and} \ xy \notin E \quad \hbox{and} \quad M_{xy} < 0 \ \hbox{for} \ x \neq y \ \hbox{and} \ xy \in E. $$
 Let~$d_G (x)$ = degree of~$x$ in $G$, let~$V (G)$ = vertex set of~$G$, and let~$E (G)$ = edge set of~$G$.
 Then, the \emph{Laplacian} of~$G$ is defined as~$\mathscr{L} = D_G - A_G$ where
 $$ D_G = \diag ((d_G (x))_{x \in V (G)}) \quad \hbox{and} \quad A_G = \ \hbox{the adjacency matrix}. $$
 In particular, $(A_G)_{xy} = \mathbbm{1} \{xy \in E (G)\}$ when the graph~$G$ is simple.}
\end{definition}

 Note that there is no restrictions on the diagonal entries of the matrix~$M$.
 Also, the Laplacian of~$G$ is clearly a generalized Laplacian.
\begin{lemma}[Perron-Frobenius for Laplacians \cite{5}]
\label{L5}
 Assume that~$M$ is a generalized Laplacian of a connected graph.
 Then, the smallest eigenvalue of~$M$ is simple and the corresponding eigenvector can be chosen with all entries positive.
\end{lemma}
\begin{lemma}[Courant-Fischer Formula \cite{7}]
\label{L6}
 Assume that~$Q$ is a symmetric matrix with eigenvalues~$\lambda_1 \leq \lambda_2 \leq \cdots \leq \lambda_n$ and corresponding eigenvectors~$v_1, v_2, \ldots, v_n$.
 Let~$S_k$ be the vector space generated by~$v_1, v_2, \ldots, v_k$ and~$S_0 = \{0 \}$. Then,
 $$ \lambda_k = \min \{x' Q x : \|x \| = 1, x \in S_{k - 1}^{\perp} \}. $$
\end{lemma}
\begin{lemma}[Cheeger's Inequality \cite{6}]
\label{L7}
 Assume that~$G = (V, E)$ is an undirected graph with the Laplacian~$\mathscr{L}$. Define
 $$ i (G) = \min \bigg\{\frac{|\partial S|}{|S|} : S \subset V, 0 < |S| \leq \frac{|G|}{2} \bigg\} $$
 where~$\partial S = \{uv \in E : u \in S, v \in S^c\}$. Then,
 $$ 2i (G) \geq \lambda_2 (\mathscr{L}) \geq \frac{i^2(G)}{2 \Delta (G)} \quad \hbox{where} \quad \Delta (G) = \ \hbox{maximum degree of~$G$}. $$
\end{lemma}
\begin{lemma}\label{NL8}
 Let~$Z(t)=\sum_{i,j\in[n]}\|x_i(t)-x_j(t)\|^2\wedge\epsilon^2.$ Then,~$Z$ is nonincreasing with respect to~$t$. In particular,
$$\begin{array}{rcl}
    \displaystyle Z(t)-Z(t+1)&\n \geq\n &\displaystyle 4\sum_{i=1}^n \bigg(1+|N_i(t)|\frac{\alpha_i(t)}{1-\alpha_i(t)}\mathbbm{1}\{\alpha_i(t)<1\}\bigg)\|x_i(t)-x_i(t+1)\|^2.
\end{array}$$
\end{lemma}
\begin{proof}
Let~$N_i=N_i(t)$, $N_i^*=N_i(t+1)$, $\alpha=\alpha(t)$, $x=x(t)$, $x^*=x(t+1)$ and $x_i'=\frac{1}{|N_i|}\sum_{k\in N_i}x_k$ for all~$i\in[n]$. Via the Cauchy-Schwarz inequality, we obtain \allowdisplaybreaks
\begin{align*}
    & Z(t)-Z(t+1)  = \sum_{i,j\in[n]}(\|x_i-x_j\|^2\wedge\epsilon^2-\|x_i^*-x_j^*\|^2\wedge\epsilon^2) \\
    &\hspace*{20pt} = \sum_{i=1}^n\bigg[\sum_{j\in N_i\cap N_i^*}(\|x_i-x_j\|^2-\|x_i^*-x_j^*\|^2)+\sum_{j\in N_i-N_i^*}(\|x_i-x_j\|^2-\epsilon^2)\\
    &\hspace*{50pt} + \sum_{j\in N_i^*-N_i}(\epsilon^2-\|x_i^*-x_j^*\|^2)\bigg]\\ &\hspace{20pt}\geq \sum_{i=1}^n\sum_{j\in N_i}(\|x_i-x_j\|^2-\|x_i^*-x_j^*\|^2)\\ &\hspace*{20pt} = \sum_{i=1}^n\sum_{j\in N_i}(\|x_i-x_j\|^2-\|x_i^*-x_j\|^2+\|x_i^*-x_j\|^2-\|x_i^*-x_j^*\|^2)\\
    &\hspace*{20pt} = \sum_{i=1}^n\sum_{j\in N_i}(\|x_i-x_i^*+x_i^*-x_j\|^2-\|x_i^*-x_j\|^2+\|x_i^*-x_j^*+x_j^*-x_j\|^2\\
    &\hspace*{70pt} -\|x_i^*-x_j^*\|^2)\\
    &\hspace*{20pt} =  \sum_{i=1}^n\sum_{j\in N_i}(\|x_i-x_i^*\|^2+2<x_i-x_i^*,x_i^*-x_j>+\|x_j^*-x_j\|^2\\ &\hspace*{70pt}  +2<x_i^*-x_j^*,x_j^*-x_j>)\\
    &\hspace*{20pt} =  \sum_{i=1}^n |N_i|(\|x_i-x_i^*\|^2+2<x_i-x_i^*,x_i^*-x_i'>)+\sum_{j=1}^n\sum_{i\in N_j}\|x_j^*-x_j\|^2\\
    &\hspace*{30pt} +2\sum_{j=1}^n\sum_{i\in N_j}<x_i^*-x_i+x_i-x_j^*,x_j^*-x_j>\\
    &\hspace*{20pt} =  \sum_{i=1}^n |N_i|\bigg(\|x_i-x_i^*\|^2+\frac{2\alpha_i}{1-\alpha_i}\mathbbm{1}\{\alpha_i<1\}\|x_i-x_i^*\|^2\bigg)+\sum_{j=1}^n |N_j|\|x_j^*-x_j\|^2\\ &\hspace*{30pt}  +2\sum_{j=1}^n\sum_{i\in N_j}<x_i^*-x_i,x_j^*-x_j>+2\sum_{j=1}^n\sum_{i\in N_j}<x_i-x_j^*,x_j^*-x_j>\\
    &\hspace*{20pt} = \sum_{i=1}^n |N_i|\bigg(1+\frac{2\alpha_i}{1-\alpha_i}\mathbbm{1}\{\alpha_i<1\}\bigg)\|x_i-x_i^*\|^2+\sum_{i=1}^n|N_i|\|x_i^*-x_i\|^2\\
    &\hspace*{30pt} +2\sum_{j=1}^n<x_j^*-x_j,x_j^*-x_j>+2\sum_{j=1}^n\sum_{i\in N_j-\{j\}}<x_i^*-x_i,x_j^*-x_j>\\
    &\hspace*{30pt} +2\sum_{j=1}^n |N_j|<x_j'-x_j^*,x_j^*-x_j>
    \end{align*}
     \begin{align*}
    &\hspace*{20pt} \geq \sum_{i=1}^n |N_i|\bigg(2+\frac{2\alpha_i}{1-\alpha_i}\mathbbm{1}\{\alpha_i<1\}\bigg)\|x_i-x_i^*\|^2+2\sum_{j=1}^n\|x_j^*-x_j\|^2\\
    &\hspace*{30pt}  -2\sum_{j=1}^n\sum_{i\in N_j-\{j\}}\|x_i^*-x_i\|\|x_j^*-x_j\|+2\sum_{j=1}^n |N_j|\frac{\alpha_j}{1-\alpha_j}\mathbbm{1}\{\alpha_j<1\}\|x_j^*-x_j\|^2\\
    &\hspace*{20pt} = \sum_{i=1}^n |N_i|\bigg(2+\frac{4\alpha_i}{1-\alpha_i}\mathbbm{1}\{\alpha_i<1\}\bigg)\|x_i-x_i^*\|^2+2\sum_{j=1}^n\|x_j^*-x_j\|^2\\
    &\hspace*{30pt}  +\sum_{j=1}^n\sum_{i\in N_j-\{j\}}\bigg[(\|x_i^*-x_i\|-\|x_j^*-x_j\|)^2-\|x_i^*-x_i\|^2-\|x_j^*-x_j\|^2\bigg]\\
    &\hspace*{20pt} \geq \ \sum_{i=1}^n |N_i|\bigg(2+\frac{4\alpha_i}{1-\alpha_i}\mathbbm{1}\{\alpha_i<1\}\bigg)\|x_i-x_i^*\|^2+2\sum_{j=1}^n\|x_j^*-x_j\|^2\\
    &\hspace*{30pt} -\sum_{i=1}^n\sum_{j\in N_i-\{i\}}\|x_i^*-x_i\|^2-\sum_{j=1}^n\sum_{i\in N_j-\{j\}}\|x_j^*-x_j\|^2\\
    &\hspace*{20pt} =  \sum_{i=1}^n |N_i|\bigg(2+\frac{4\alpha_i}{1-\alpha_i}\mathbbm{1}\{\alpha_i<1\}\bigg)\|x_i-x_i^*\|^2+2\sum_{j=1}^n\|x_j^*-x_j\|^2\\
    &\hspace*{30pt} -\sum_{i=1}^n(|N_i|-1)\|x_i^*-x_i\|^2-\sum_{j=1}^n(|N_j|-1)\|x_j^*-x_j\|^2\\
    &\hspace*{20pt}  =  \sum_{i=1}^n 4 \bigg(1+|N_i|\frac{\alpha_i}{1-\alpha_i}\mathbbm{1}\{\alpha_i<1\}\bigg)\|x_i-x_i^*\|^2
\end{align*}
This completes the proof.
\end{proof}
\begin{lemma}
\label{L10}
 Assume that~$Q$ is a real square matrix and that~$V$ is invertible such that the matrix~$VQ = \mathscr{L}$ is the Laplacian of some connected graph.
 Then, 0 is a simple eigenvalue of $Q'Q$ corresponding to the eigenvector~$\mathbbm{1} = (1, 1,\ldots, 1)'$.
 In particular, we have
 $$ \lambda_2 (Q'Q) = \min \{x'Q'Qx : \|x \| = 1 \ \hbox{and} \ x \perp \mathbbm{1} \}. $$
\end{lemma}
\begin{proof}
 To begin with, observe that
 $$ Q'Qx = 0 \ \iff \ Qx = 0 \ \iff \ \mathscr{L}x = 0. $$
 Recall that a real symmetric matrix is diagonalizable, and that its algebraic multiplicity = its geometric multiplicity.
 Since~$\mathscr{L}$ is positive semi-definite and has an eigenvalue 0 corresponding to the eigenvector~$\mathbbm{1}$, by Lemma \ref{L5}, 0 is a simple eigenvalue of~$\mathscr{L}$.
 Hence, by the above relation between~$\mathscr{L}$ and~$Q'Q$, the matrix~$Q'Q$ has a simple eigenvalue 0 corresponding to the eigenvector~$\mathbbm{1}$.
 Since in addition
 $$ x'Q'Qx = \|Qx \|^2 \geq 0, $$
 the matrix~$Q'Q$ is positive semi-definite.
 Finally, applying~Lemma~\ref{L6}, we get
 $$ \lambda_2 (Q'Q) = \min \{x'Q'Qx : \|x \| = 1 \ \hbox{and} \ x \perp \mathbbm{1} \}. $$
 This completes the proof.
\end{proof}

 Now, we are ready to investigate several conditions under which, for any~$\delta > 0$, every component of a profile is~$\delta$-trivial in finite time.
\begin{theorem}
\label{T3}
 Assume that~$\limsup_{t \to \infty} \max_{i \in [n]} \alpha_i (t) < 1$.
 Then, for any~$\delta > 0$, every component of a profile is~$\delta$-trivial in finite time, i.e.,
 $$ \tau_{\alpha, \delta} := \inf \{t \geq 0 : \hbox{every component of $\mathscr{G} (t)$ is $\delta$-trivial} \} < \infty. $$
\end{theorem}
\begin{proof}
 If every component of~$\mathscr{G} (t)$ is~$\delta$-trivial, we are done.
 Now, assume that~$\mathscr{G} (t)$ has a $\delta$-nontrivial component.
 Without loss of generality, we may assume that~$\mathscr{G} (t)$ is connected;
 if not, we can restrict to a~$\delta$-nontrivial component.
 For~$\mathbbm{1} \in \mathbf{R^n}$ and~$W = \Span(\{\mathbbm{1} \})$, $\mathbf{R^n}=W\oplus W^\perp$.
 Then, write
 $$ x (t) = \left[c_1 \mathbbm{1} \,| \,c_2 \mathbbm{1} \,| \,\cdots \,| \,c_d \mathbbm{1} \right] +
            \left[\hat{c}_1 u^{(1)} \,| \,\hat{c}_2 u^{(2)} \,| \,\cdots \,| \,\hat{c}_d u^{(d)} \right] $$
 where $c_i$ and $\hat{c}_i$ are constants and $u^{(i)} \in \mathbbm{1}^\perp$ is a unit vector for all~$i \in [d]$.
 $$ \hbox{Claim:} \quad \sum_{k = 1}^d \,\hat{c}_k^2 > \frac{\delta^2}{2}. $$
 Assume by contradiction that this is not the case.
 Then, for any~$i, j \in [n]$,
 $$ \begin{array}{l}
    \displaystyle \|x_i (t) - x_j (t)\|^2 =
    \displaystyle \sum_{k = 1}^d \,\hat{c}_k^2 (u^{(k)}_i - u^{(k)}_j)^2 \vspace*{-4pt} \\ \hspace*{40pt} \leq
    \displaystyle \sum_{k = 1}^d \,\hat{c}_k^2 \ 2 ((u^{(k)}_i)^2 + (u^{(k)}_j)^2) \leq 2 \,\sum_{k = 1}^d \,\hat{c}_k^2 \leq \delta^2, \end{array} $$
 contradicting the~$\delta$-nontriviality of~$\mathscr{G}(t)$.
 Let~$B (t) = \diag (\alpha(t)) + (I - \diag (\alpha(t))) A (t)$. Then,
 $$ x (t) - x (t + 1) = (I - B (t)) \,x(t) = \left[\hat{c}_1 (I - B (t)) u^{(1)} \,| \,\cdots \,| \,\hat{c}_d (I - B(t)) u^{(d)} \right],$$
 from which it follows that
 $$ \sum_{i = 1}^n \,\|x_i (t) - x_i (t + 1) \|^2 = \sum_{j = 1}^d \,\hat{c}_j^2 \|(I - B (t)) u^{(j)} \|^2. $$
 Now, observe that
 $$ I - B (t) = (I - \diag (\alpha (t)))(I + D (t))^{-1} \mathscr{L} $$
 where~$\mathscr{L}$ is the Laplacian of~$\mathscr{G} (t)$ and~$D (t)$ is diagonal with~$D_{ii} (t) = d_i (t)$, the degree of vertex~$i$.
 Assume that~$\alpha_i (t) < 1$ for all~$i \in [n]$.
 Then, $I - \diag (\alpha (t))$ is invertible, and according to~Lemmas~\ref{L7} and~\ref{L10},
 $$ \begin{array}{l}
    \displaystyle \|(I - B(t)) u^{(j)} \|^2 =
    \displaystyle u^{(j)'} (I-B(t))' (I-B(t))u^{(j)} \geq
    \displaystyle \lambda_2 ((I-B(t))' (I-B(t))) \vspace*{8pt} \\ \hspace*{25pt} =
    \displaystyle \lambda_2 \bigg(\mathscr{L} \,\diag\bigg(\bigg(\bigg(\frac{1-\alpha_i(t)}{1 + d_i(t)}\bigg)^2 \bigg)_{i = 1}^n \bigg) \mathscr{L} \bigg) \vspace*{8pt} \\ \hspace*{25pt} \geq
    \displaystyle \bigg(\frac{1 - \max_{i \in [n]} \alpha_i (t)}{n} \bigg)^2 \lambda_2 (\mathscr{L}^2) =
    \displaystyle \bigg(\frac{1 - \max_{i \in [n]} \alpha_i (t)}{n} \bigg)^2 \lambda_2^2 (\mathscr{L}) \vspace*{8pt} \\ \hspace*{25pt} >
    \displaystyle  \frac{4 (1 - \max_{i \in [n]} \alpha_i (t))^2}{n^8} \end{array} $$
 where we used that
 $$ \lambda_2 (\mathscr{L}) \geq \frac{i^2 (\mathscr{G}(t))}{2 \Delta(\mathscr{G} (t))} > \frac{(2/n)^2}{2n} = \frac{2}{n^3}. $$
 In particular, we obtain
 $$ \sum_{i = 1}^n \,\|x_i (t) - x_i (t + 1) \|^2 > \frac{2 \delta^2(1 - \max_{i \in [n]} \alpha_i (t))^2}{n^8}. $$
 Since~$\limsup_{t \to \infty} \max_{i \in [n]} \alpha_i (t) <1$, there exists $(t_k)_{k \geq 1} \subset \mathbf{N}$ strictly increasing such that
 $$ \max_{i \in [n]} \alpha_i (t_k) \leq \gamma < 1 \quad \hbox{for some} \quad \gamma \quad \hbox{and for all} \quad k \geq 1. $$
 Now, let~$\tau = \tau_{\alpha, \delta}$.
 By Lemma \ref{NL8}, for all~$m \geq 1$,
\begin{align*}
  n^2 \epsilon^2 \ > \ & Z (0) \geq Z (0) - Z (m) = \sum_{t = 0}^{m - 1} \,(Z (t ) - Z (t+1)) \\
                 \geq \ & \sum_{t = 0}^{m - 1} \sum_{i \in [n], \alpha_i (t) < 1} 4 \bigg(1+|N_i(t)| \frac{\alpha_i(t)}{1 - \alpha_i (t)}  \bigg) \|x_i (t) - x_i (t + 1) \|^2 \\
                 \geq \ & 4 \sum_{t = 0}^{m - 1} \sum_{i \in [n], \alpha_i (t) < 1} \|x_i (t) - x_i (t + 1) \|^2. \label{st}\tag{$\ast$}
\end{align*}
 Now, assume by contradiction that~$\tau = \infty$.
 Letting~$m \to \infty$, we get
\begin{align*}
  n^2 \epsilon^2 \ \geq \ & 4 \sum_{t = 0}^{\infty} \sum_{i \in [n], \alpha_i (t) < 1} \|x_i (t) - x_i (t + 1) \|^2 \\
                 \ \geq \ & 4 \sum_{t \geq 0, \max_{i \in [n]} \alpha_i (t) < 1} \frac{2 \delta^2 (1 - \max_{i \in [n]} \alpha_i(t))^2}{n^8} \\
                 \ \geq \ & \sum_{k \geq 1} \frac{8 \delta^2 (1 - \max_{i \in [n]} \alpha_i (t_k))^2}{n^8} \geq \sum_{k \geq 1} \frac{8 \delta^2 (1 - \gamma)^2}{n^8} = \infty,
\end{align*}
 a contradiction.
 This completes the proof.
\end{proof}

 From Theorem~\ref{T3}, if~$\limsup_{t \to \infty} \max_{i \in [n]} \alpha_i (t) < 1$, then~$\tau_{\alpha, \delta} < \infty$.
 Thus, if~$\mathscr{G} (\tau_{\alpha, \delta})$ is connected for some~$0 < \delta \leq \epsilon$, then by Theorem~\ref{T1}, a consensus is reached eventually.
 The main parts of the proof of Theorem~\ref{T3} resemble the ones in the proof of Theorem~2 in~\cite{2}.
 The similarities between the proofs consist in the derivation of a lower bound for
 $$ \sum_{i = 1}^n \,\|x_i (t) - x_i (t + 1) \|^{2} $$
 by restricting to a~$\delta$-nontrivial component and then choosing a bounded function to construct an inequality involving the sum.
 The main difference is that~Theorem~\ref{T3} assumes that
\begin{equation}
\label{eq:limsup}
  \limsup_{t \to \infty} \max_{i \in [n]} \alpha_i (t) < 1
\end{equation}
 to ensure that the smallest eigenvalue of~$(I - B (t))'(I - B (t))$ is simple, but Theorem~2 in~\cite{2} has no such assumptions since~\eqref{eq:limsup}
 automatically holds if~$\alpha (t) = \Vec{0}$ for all~$t \geq 0$.
 Theorem~2 in~\cite{2} states that the termination time of the synchronous HK model is independent of~$d$ and bounded from above.
 In fact, the result is a special case of the following corollary.
\begin{corollary}
\label{Co1}
 Assume that~$\sup_{t \in \mathbf{N}} \max_{i \in [n]} \alpha_i (t) < 1$.
 Then, $\tau_{\alpha, \delta}$ is bounded from above.
 Also, letting~$\tau_m = \tau_{\alpha, \epsilon / m}$ for~$m \geq 4$, there is no interactions between any two components of~$\mathscr{G} (t)$ at the next time step
 for some $M \geq 4$ and for all~$t \geq \tau_M$, i.e.,
 $$ \mathscr{G} (t) = \mathscr{G} (\tau_M) \quad \hbox{for some} \quad M \geq 4 \quad \hbox{and for all} \quad t \geq \tau_M. $$
 Hence, $x$ in \eqref{mHK} is asymptotically stable.
\end{corollary}
\begin{proof}
 Because
 $$ \limsup_{t \to \infty} \max_{i \in [n]} \alpha_i (t) \leq \sup_{t\in\mathbf{N}} \max_{i \in [n]} \alpha_i (t) < 1, $$
 it follows from Theorem \ref{T3} that~$\tau < \infty$.
 For~$\tau\geq1$, setting~$m = \tau$ in~\eqref{st}, we get
 $$ \begin{array}{rcl}
    \displaystyle  n^2 \epsilon^2 & \n > \n &
    \displaystyle 4 \sum_{t = 0}^{\tau - 1} \sum_{i \in [n], \alpha_i (t) < 1} \|x_i (t) - x_i (t + 1) \|^2 \vspace*{4pt} \\ & \n = \n &
    \displaystyle 4 \sum_{t = 0}^{\tau - 1} \sum_{i = 1}^n \,\|x_i (t) - x_i (t + 1) \|^2 \geq
    \displaystyle 4 \sum_{t = 0}^{\tau - 1} \ \frac{2 \delta^2 (1 - \max_{i \in [n]} \alpha_i (t))^2}{n^8} \vspace*{8pt} \\ & \n \geq \n &
    \displaystyle \frac{8 \tau \delta^2(1 - \sup_{t \in \mathbf{N}} \max_{i \in [n]} \alpha_i(t))^2}{n^8}, \end{array} $$
 from which it follows that
 $$ \tau < \frac{n^{10}}{8 (1 - \sup_{t \in \mathbf{N}} \max_{i \in [n]} \alpha_i (t))^2} \bigg(\frac{\epsilon}{\delta} \bigg)^2. $$
 Hence,~$\tau$ is bounded from above.
 To show the asymptotic stability of~$x$, we first observe that~$\tau_m$ is finite and nondecreasing with respect to~$m$. For all~$0<\delta\leq\epsilon/4$ and~$t\geq 0$, assume that every component of~$\mathscr{G}(t)$ is~$\delta$-trivial. Then, the following three conditions are equivalent:
 \begin{enumerate}
     \item\label{q1} Some component of~$\mathscr{G}(t+1)$ is~$\delta$-nontrivial.\vspace*{2pt}
     \item\label{q2} Some components of~$\mathscr{G}(t)$ interact at time~$t+1$.\vspace*{2pt}
     \item\label{q3} Some component of~$\mathscr{G}(t+1)$ is~$\epsilon/2$-nontrivial.
 \end{enumerate}
 It is clear that~\ref{q1} $\Rightarrow$ \ref{q2} and~\ref{q3} $\Rightarrow$ \ref{q1}; therefore we show~\ref{q2} $\Rightarrow$ \ref{q3}.
 \begin{proof}[Proof of \ref{q2} $\Rightarrow$ \ref{q3}]
 Let the convex hull of a component~$G$ be
 $$ Cv(G) = C(\{x_j : j \in V(G) \}). $$
 The fact that some components of~$\mathscr{G} (t)$ interact at time~$t+1$ implies that there exist
 $$ i, j \in [n] \quad \hbox{with} \quad ij \in \mathscr{E} (t + 1), \quad \quad i\in V(G_{\Tilde{i}})\quad \hbox{and}\ j\in V(G_{\Tilde{j}}) $$
 for some distinct components~$G_{\Tilde{i}}$ and~$G_{\Tilde{j}}$ of~$\mathscr{G} (t)$.
 Therefore,
 $$ x_i (t + 1) \in Cv (G_{\Tilde{i}}) \quad \hbox{and} \quad x_j (t + 1) \in Cv (G_{\Tilde{j}}). $$
 Hence,
\begin{align*}
 \epsilon \ < \ & \|x_i (t) - x_j (t) \| \\
          \ \leq \ & \|x_i (t) - x_i (t + 1) \| + \|x_i (t + 1) - x_j (t + 1) \| + \|x_j (t + 1) - x_j (t) \| \\
          \ \leq \ & \delta + \|x_i (t + 1) - x_j (t + 1) \| + \delta = \|x_i (t + 1) - x_j (t + 1) \| + 2 \delta. \end{align*}
This implies that
$$ \|x_i (t + 1) - x_j (t + 1) \| > \epsilon - 2 \delta \geq \epsilon - 2 \cdot \frac{\epsilon}{4} = \frac{\epsilon}{2} \quad \hbox{for all} \quad 0 < \delta \leq \frac{\epsilon}{4} $$
so the component of~$\mathscr{G} (t + 1)$ containing~$ij$ is~$\epsilon/2$-nontrivial.
 \end{proof}

 Let
 $$ A_m = \{t \in [\tau_m, \tau_{m + 1}) : \hbox{some component of $\mathscr{G}(t)$ is $\epsilon/m$-nontrivial} \} $$
 and~$t_m = \inf A_m$.
 $$\hbox{Claim:}\quad \hbox{the set}\ \mathscr{A} := \{t_k : A_k \neq \varnothing \} \  \hbox{is finite}.$$
For~$t_m \in \mathscr{A}$, since some component of~$\mathscr{G} (t_m)$ is~$\epsilon/m$-nontrivial and all components of~$\mathscr{G} (t_m - 1)$
 are~$\epsilon/m$-trivial, by~\ref{q1}$\Rightarrow$ \ref{q3}, some component of~$\mathscr{G} (t_m)$ is~$\epsilon/2$-nontrivial.
 Using~\eqref{st} and letting~$m \to \infty$, we get
\begin{align*}
  n^2 \epsilon^2 \ \geq \ & 4 \sum_{t \geq 0} \sum_{i \in [n], \alpha_i (t) < 1} \|x_i (t) - x_i (t + 1) \|^2 =
                            4 \sum_{t \geq 0} \,\sum_{i = 1}^n \,\|x_i (t) - x_i (t + 1) \|^2 \\
                 \ \geq \ & 4 \sum_{t \in \mathscr{A}} \ \frac{2 (\epsilon/2)^2 (1 - \max_{i \in [n]} \alpha_i (t))^2}{n^8} \\
                 \ \geq \ & |\mathscr{A}| \ \frac{8 (\epsilon/2)^2 (1 - \sup_{t \geq 0} \max_{i \in [n]} \alpha_i (t))^2}{n^8}, \end{align*}
 from which it follows that
 $$ |\mathscr{A}| \leq \frac{n^{10}}{2(1 - \sup_{t \geq 0} \max_{i \in [n]} \alpha_i (t))^2}. $$
 Hence, the set~$\mathscr{A}$ is finite.
 By the fact that~$\mathscr{A}$ is finite and that~\ref{q2}$\Rightarrow$ \ref{q1}, there is no interactions between any two components of~$\mathscr{G} (s)$ at the next time step for some~$M \geq 4$ and for all~$s \geq \tau_M$. Hence, we deduce that
 every component of~$\mathscr{G} (\tau_M)$ is an independent system.
 Since in addition
 $$ \limsup_{t \to \infty} \beta_t \leq \sup_{t \in \mathbf{N}} \beta_t \leq \sup_{t \in \mathbf{N}} \max_{i \in [n]} \alpha_i (t) < 1, $$
 by Theorem~\ref{T1}, $x$ in~\eqref{mHK} is asymptotically stable.
\end{proof}

 Note that the upper bound for~$\tau$ is independent of $d$, and~\eqref{mHK} reduces to the synchronous HK
 model if~$\alpha (t) = \Vec{0}$ for all~$t \geq 0$.
 Since~$\sup_{t \geq 0} \max_{i \in [n]} \alpha_i (t) < 1$ automatically holds if~$\alpha (t) = \Vec{0}$ at all times, $\mathscr{G} (s) = \mathscr{G} (\tau_M)$
 for some~$M \geq 4$ and for all~$s \geq \tau_M$.
 This shows that~$\mathscr{G} (\tau_M + 1)$ is a steady state and that the termination time of the synchronous HK model is bounded from above.

\section{Conclusion}
 The mixed HK model covers both the synchronous and the asynchronous HK models, and is therefore more general and more complicated.
 At each time step, each agent can choose its degree of stubbornness and mix its opinion with the average opinion of its neighbors.
 Agents with the same opinion may depart later, depicting the changeability of agents, which is closer to real world circumstances. Given the givens, make it more difficult to reach asymptotic stability or a steady state.
 However, under some conditions, not only does the asymptotic stability hold, but also a consensus can be achieved.

\medskip
Received October 2020; revised December 2020.
\medskip

\end{document}